\documentclass[a4paper,12pt]{article}

\usepackage[utf8]{inputenc}  
\usepackage[T1]{fontenc}     
\usepackage[english]{babel}  

\usepackage{geometry}    
\usepackage{graphicx} 
\usepackage{amsmath}
\usepackage{amssymb}
\usepackage{amsthm}
\usepackage{mathtools}
\usepackage[dvipsnames]{xcolor}
\usepackage{enumitem}
\usepackage{pifont}
\usepackage{hyperref}    
\usepackage{xcolor}

\usepackage{tikz}
\usetikzlibrary{matrix,calc,arrows.meta}
\usepackage{adjustbox}

\newtheorem{theorem}{Theorem}[section]
\newtheorem*{theorem*}{Main Theorem}
\newtheorem*{theoremA}{Theorem A}

\newtheorem{lemma}[theorem]{Lemma}
\newtheorem{coroll}[theorem]{Corollary}
\theoremstyle{definition}
\newtheorem{definition}[theorem]{Definition}
\newtheorem{remark}[theorem]{Remark}

\newtheorem{conj}[theorem]{Conjecture}

\newtheorem{openproblem}[theorem]{Open problem}

\usepackage{tikz-cd}
\usepackage{chronology}
\usepackage{epic}
\usepackage{pict2e}
\usepackage{ulem}
\newcommand{\redsout}{\bgroup\markoverwith{\textcolor{red}{\rule[0.5ex]{2pt}{2pt}}}\ULon}

\title{\bf  On the Hamiltonicity of generating graphs\\ of almost simple groups}
\author{LUIGI IORIO}
\date{}

\begin{document}

\maketitle

\begin{abstract}
\noindent The generating graph $\Gamma(G)$ of a finite group $G$ has vertex set $G\setminus\{1\}$, and two distinct vertices are adjacent if and only if they generate $G$. Breuer, Guralnick, Lucchini, Maróti and Nagy [Bull. Lond. Math. Soc. 42 (2010), 621–633] conjectured that, for every finite group $G$ with at least four elements, $\Gamma(G)$ contains a Hamiltonian cycle if and only if every proper quotient of $G$ is cyclic. They proved their conjecture for sufficiently large almost simple groups with alternating socle and for all almost simple groups with sporadic socle. In this paper, we complete the asymptotic picture for almost simple groups by proving the conjecture for sufficiently large almost simple groups with socle of Lie type. 
\end{abstract}

\medskip\medskip

\noindent {\bf Keywords}
\begin{sloppypar}
\noindent almost simple groups; simple groups of Lie type; probabilistic group generation; \mbox{generating graphs;} Hamiltonian cycles; monolithic groups
\end{sloppypar}

\bigskip
\noindent {\bf 2020 Mathematics Subject Classification}\\
{\it Primary}: 05C25, 20D60
{\it Secondary}: 05C45, 20P05, 20D06

\section{Introduction}
Questions concerning group generation have attracted considerable interest, both in past and present times. Perhaps the best-known result on this topic is that every finite simple group is $2$-generated (that is, it can be generated by two elements). This result, which is elementary for cyclic groups of prime order and for alternating groups, was established for groups of Lie type by Steinberg in 1962 (see \cite{Steinberg}), and for sporadic groups by Aschbacher and Guralnick in 1984 (see \cite{Aschbacher}). The literature on the subject has subsequently developed in several directions, ranging from investigating particular generating pairs with specific properties to estimating the total number of generating pairs. Regarding the latter problem, one of the most interesting achievements was the affirmative resolution of a conjecture posed in 1969 by Dixon, according to which generating pairs form an overwhelming majority: for a finite simple group, the probability that two elements, chosen independently and uniformly at random, generate the whole group tends to $1$ as the order of the group tends to infinity (for the proof, see \cite{Dixon}, \cite{KantorLubotzky} and \cite{LiebeckShalev}).

Already in his 1962 paper \cite{Steinberg}, Steinberg wondered whether a property stronger than \mbox{$2$-generation} might hold: is it true that, in a finite simple group, every element forms a generating pair with some other element? An affirmative answer was obtained in the late 1990s, independently by Stein in \cite{Stein} and by Guralnick and Kantor in \cite{GurKant2000}. In general, a group satisfying this property is said to be $(3/2)$-generated. In 2008, Breuer, Guralnick and Kantor in \cite{BGK} conjectured that, for finite groups, $(3/2)$-generation is equivalent to having all proper quotients cyclic. That conjecture was proved more recently, in 2021, by Burness, Guralnick and Harper in \cite{Annals}. Their proof proceeds through a third equivalent property: for every choice of two elements, there exists a third element that forms a generating pair with each of them, namely, the spread of the group is at least~$2$ (the spread of a group is the largest non-negative integer $k$ such that, for every $k$-tuple of elements, there exists another element that forms a generating pair with each of them). Thus, for a finite group~$G$, denoting by~$s(G)$ the spread of~$G$, we have the following equivalences.
$$
\begin{array}{c}
G/N \mbox{ is cyclic for every } \{1\}\neq N\trianglelefteq G\\[3pt]
\Updownarrow \\[3pt]
s(G)\geq 2\\[3pt]
\Updownarrow \\[3pt]
G \mbox{ is \((3/2)\)-generated }
\end{array}
$$
The property \(s(G)\geq 2\) may be interpreted as a local pattern among generating pairs, and it is therefore natural to consider whether generating pairs could be organized into global patterns of some kind. Here, a different point of view comes into play, offered by graph theory.
Given a group $G$, the generating graph $\Gamma(G)$ of $G$ is defined by taking as vertices the non-identity elements of $G$, and declaring that two vertices are adjacent if and only if they form a generating pair of $G$. It is clear that many results concerning the generation of a group can be expressed in terms of properties of its generating graph. For instance, saying that a group is generated by two distinct elements is equivalent to saying that its generating graph has at least one edge; also, the $(3/2)$-generation property is equivalent to saying that the generating graph has no isolated vertices, that is, every vertex has degree at least $1$; and again, having spread at least $2$ means that every pair of vertices has a common neighbour.
Summarizing, for a finite group $G$, denoting by $\delta(\Gamma(G))$ the minimum of the degrees of the vertices of $\Gamma(G)$ and by $N_{\Gamma(G)}(v)$ the set of neighbours of a vertex $v$, the Burness-Guralnick-Harper result in terms of generating graphs translates as follows:
$$
\begin{array}{c}
G/N \mbox{ is cyclic for every } \{1\}\neq N\trianglelefteq G\\[3pt]
\Updownarrow \\[3pt]
\forall v,w \mbox{ vertices of } \Gamma(G), \,N_{\Gamma(G)}(v) \cap N_{\Gamma(G)}(w) \neq \emptyset \\[3pt]
\Updownarrow \\[3pt]
\Gamma(G) \mbox{ has no isolated vertices } \mbox{(that is, \( \delta(\Gamma(G)) \geq 1\))}
\end{array}
$$
Observe that, if we assume that $G$ has at least $4$ elements, then to this chain of equivalences we can also add the condition that every vertex of $\Gamma(G)$ has degree at least $2$, that is, $\delta(\Gamma(G))\geq 2$.

In 2010, when the Breuer-Guralnick-Kantor conjecture had already been posed, Breuer, Guralnick, Lucchini, Maróti and Nagy took a further step within the graph-theoretic framework. They envisaged the possibility that the condition of having all proper quotients cyclic not only implied that every vertex had a neighbour, but also that some generating pairs could be organized into a specific pattern, more precisely into a Hamiltonian cycle, that is, a closed path passing through all the vertices of the graph. Thus, they posed the following conjecture.
\begin{conj}[\cite{5authors}, Conjecture 1.6 - Breuer, Guralnick, Lucchini, Maróti, Nagy]\label{MarotiConj} \ \\
    Let $G$ be a finite group with at least four elements. Then the graph $\Gamma(G)$ contains a Hamiltonian cycle if and only if $G/N$ is cyclic for all non-trivial normal subgroups $N$ of~$G$.
\end{conj}
Note that one implication is immediate. If $\Gamma(G)$ is Hamiltonian (namely, it has a Hamiltonian cycle), then $\delta(\Gamma(G)) \geq 1$. Thus, if we fix a non-trivial normal subgroup $N$ and an $x\in N\setminus\{1\}$, then there exists $y\in G$ such that $G=\langle x,y\rangle$, hence $G/N = \langle yN \rangle$. Therefore, the real content of the conjecture lies in the converse implication.

If this conjecture were true, the Hamiltonicity of $\Gamma(G)$ would be added to the equivalent conditions established by the Burness-Guralnick-Harper theorem. Thus, it would imply the striking fact that, for generating graphs, the only condition determining the existence of a Hamiltonian cycle would be the absence of isolated vertices, which is clearly a minimal necessary requirement. 
In the same paper in which Conjecture~\ref{MarotiConj} was posed, the authors proved it for several classes of finite groups, including: solvable groups; sufficiently large simple groups; sufficiently large almost simple groups with alternating socle; almost simple groups with sporadic socle; wreath products of the form $S\wr C_m$, where $S$ is a sufficiently large non-abelian simple group and $m$ is a prime power. In this paper, we prove that Conjecture~\ref{MarotiConj} also holds for sufficiently large almost simple groups with socle of Lie type, thereby completing the asymptotic picture for the class of almost simple groups. Clearly, for an almost simple group $G$ with socle $S$, the condition that all proper quotients are cyclic reduces to requiring $G/S$ to be cyclic. Thus, our main result is the following theorem.
\begin{theoremA}
    Let $G$ be a finite almost simple group with socle $S$ of Lie type such that $G/S$ is cyclic. If $G$ is sufficiently large, or equivalently if $S$ is sufficiently large, then the generating graph $\Gamma(G)$ of $G$ is Hamiltonian.
\end{theoremA}

The proof of Theorem A relies heavily on two recent probabilistic estimates for the number of generating pairs in almost simple groups (Theorem 1.1 in \cite{FULMAN} and Corollary 9 in \cite{uniserial}), together with an explicit combinatorial construction of the required Hamiltonian cycle.
Section 2 collects some preliminary material needed for our purposes, much of which will be familiar to the expert reader. In Section 3, we prove Theorem A through a sequence of subsections. Finally, in Section 4, we propose some open problems for future investigation.

\section{Preliminaries}

The aim of this section is to collect, for the reader’s convenience, terminology, notation and results that will be used frequently throughout the paper, while also deriving the consequences that will be useful for our purposes. The contents will include elementary notions from graph theory, basic facts about almost simple groups, and some recent probabilistic estimates concerning the $2$-generation of certain groups.

\subsection{Graph-theoretic tools}

In our context, a graph is a pair $\Delta=(V,E)$, where $V$ is a set whose elements are called vertices, and $E$ is a set of two-element subsets of $V$, called edges. If $v,w$ are distinct elements of $V$ and $\{v,w\}\in E$, we say that there is an edge between $v$ and $w$, or equivalently that $v$ (resp. $w$) is adjacent to $w$ (resp. $v$) in $\Delta$; moreover, $v$ and $w$ are called the endpoints of the edge $\{v,w\}$. If $v\in V$, by $N_{\Delta}(v)$ we denote the subset of vertices that are adjacent to $v$ in $\Delta$; these vertices are also called the neighbours of $v$ in $\Delta$. A graph is said to be finite if and only if its vertex set is finite. Note that, according to our definition, a graph has no loops (edges whose endpoints are the same vertex) and no multiple edges (there is at most one edge between any two vertices).

\begin{definition}[Generating graph of a group]
    Let $G$ be a group. The generating graph $\Gamma(G)$ of $G$ is the graph whose vertices are the non-identity elements of $G$ and whose edges are the two-element subsets of $G\setminus\{1\}$ that generate $G$. For example, the generating graph of $S_3$ has vertex set $S_3\setminus\{1\}$, and its edges are the two-element subsets consisting either of two transpositions or of a transposition and a $3$-cycle.

\end{definition}

\begin{definition}[Degree of a vertex]\label{degree}
    If $\Delta$ is a graph and $x$ is a vertex of $\Delta$, we denote by $d_\Delta(x)$ the degree of $x$, that is, the number of vertices adjacent to $x$.

In the context of the generating graph $\Gamma(G)$ of a finite group $G$, if $A\subseteq G$ and $x$ is a vertex of $\Gamma(G)$, we denote by $d_A(x)$ the number of vertices adjacent to $x$ in $\Gamma(G)$ and belonging to $A$, that is
 $$
 d_A(x) \coloneqq \big| N_{\Gamma(G)}(x)\cap A \big| \, .
 $$
 In particular, even when $1\in A$, $d_A(x)=d_{A\setminus\{1\}}(x)$, since $1$ is a not a vertex of $\Gamma(G)$. 

For later use, it is worth emphasizing that our definitions give rise to a potential discrepancy between $d_A(x)$ and the number of pairs
$(x,a)\in\{x\}\times A$ that generate $G$. Indeed, $d_A(x)$ cannot, by definition, count either $1$ (because $1$ is not a vertex of
$\Gamma(G)$) or~$x$ (because loops are not allowed in our graphs). In general, we only have
$$
d_A(x) = \big| \{(x,a)\in\{x\}\times A :
a\neq 1,\ a\neq x \mbox{ and } \langle x,a\rangle=G\} \big| \, .
$$
This discrepancy arises precisely when $x$ generates $G$ and either
$1\in A$ or $x\in A$. However, \textit{if $G$ is non-cyclic}, then
$$
d_A(x) = \big| \{(x,a)\in\{x\}\times A :
\langle x,a\rangle=G\} \big| \, ,
$$
and the discrepancy vanishes. Moreover, in this case, we also have
$$
|A|-d_A(x) = \big|\{(x,a)\in\{x\}\times A:
\langle x,a\rangle\neq G\}\big| \, ,
$$
thereby ensuring that $|A|-d_A(x)$ precisely counts the number of pairs in $\{x\}\times A$ that do not generate $G$.
\end{definition}

\begin{definition}[Path, Hamiltonian graph]
Let $\Delta=(V,E)$ be a graph. A path in $\Delta$ is a sequence
$$
P=(v_0,v_1,\ldots,v_k)
$$
of distinct vertices of $V$ such that $\{v_{i-1},v_i\}\in E$ for every $i\in\{1,\ldots,k\}$; $v_0$ and $v_k$ are called the endvertices of $P$.
If $k\geq2$ and $\{v_0,v_k\}\in E$, then $P$ is called a cycle. A cycle containing every vertex of $V$ is called a Hamiltonian cycle. A graph is called Hamiltonian if it contains a Hamiltonian cycle.
\end{definition}

In our constructions, we shall need a well-known result of Philip Hall on perfect matchings in bipartite graphs. The original reference for it is the 1935 paper of Hall (\cite{Hall}, Theorem~1), while the same result stated in the language of graph theory can be found in \cite{Diestel}, Theorem~2.1.2. In any case, the form in which we shall use it is presented below.

\begin{definition}[Bipartite graph]
A graph $\Delta=(V,E)$ is called bipartite if its vertex set $V$ can be partitioned into two disjoint subsets $X$ and $Y$ such that every edge of $\Delta$ has one endpoint in $X$ and the other in $Y$. The pair $(X,Y)$ is called a bipartition of $\Delta$ (as well as $(Y,X)$, here the order does not matter), and $X$ and $Y$ are called its parts.
\end{definition}

\begin{definition}[Perfect matching]
Let $\Delta=(V,E)$ be a graph. A perfect matching in~$\Delta$ is a subset $M$ of $E$ such that every vertex in $V$ is an endpoint of exactly one edge in $M$. In particular, if $\Delta$ is a bipartite graph with bipartition $(X,Y)$ such that $|X|=|Y|$, then a perfect matching in $\Delta$ describes a bijection between $X$ and $Y$ (by mapping each vertex $x\in X$ to the unique vertex $y\in Y$ such that $\{x,y\}\in M$).
\end{definition}

\begin{theorem}[Hall's marriage theorem]\label{Hall}
Let $\Delta$ be a finite bipartite graph with bipartition $(X,Y)$ such that $|X|=|Y|$. Then $\Delta$ has a perfect matching if and only if
$$
|N_\Delta(A)|\geq |A|
$$
for every non-empty subset $A$ of $X$, where $N_\Delta(A)$ denotes the set of vertices (of $Y$) adjacent to at least one vertex of $A$.
\end{theorem}

\begin{sloppypar}
\begin{coroll}\label{Hallcor}
Let $\Delta$ be a finite bipartite graph with bipartition $(X,Y)$ such that \mbox{$|X|=|Y| \eqqcolon N$.} Assume that every vertex of $\Delta$ has degree at least $N/2$. Then~$\Delta$ has a perfect matching.
\end{coroll}
\end{sloppypar}

\begin{proof}
By Hall's Marriage Theorem, it is enough to prove that $|N_\Delta(A)|\geq |A|$
for every non-empty subset $A$ of $X$.

Suppose first that $|A|\leq N/2$.
Choose $x\in A$. Since $d_\Delta(x)\geq N/2$, we have
$$
|N_\Delta(A)|\geq d_\Delta(x)\geq\frac{N}{2}\geq |A| \, .
$$

Suppose now that $|A| > N/2$. If there existed a vertex $y\in Y\setminus N_\Delta(A)$, then all the neighbours of $y$ would belong to $X\setminus A$. Hence,
$$
d_\Delta(y)\leq |X\setminus A|
=N-|A|
<\frac{N}{2} \, ,
$$
contradicting the hypothesis. Therefore, $N_\Delta(A)=Y$, and thus $|N_\Delta(A)|=N\geq |A|$.

Hence, Hall's condition holds and $\Delta$ has a perfect matching.
\end{proof} 

\subsection{Simple groups of Lie type}
We shall focus on finite simple groups of Lie type.
For our purposes, it will suffice to recall that, once a (not necessarily unique) Lie-type description of such a group has been fixed, two positive integer parameters are attached to that description: the Lie rank $r$ and the field size $q$, where $q$ is a prime power. The results of Harper and Quick in \cite{uniserial} that we shall use later are formulated in terms of $r$ and $q$, so, to avoid compatibility issues, we adopt their conventions --- the same as those in \cite{Liebeck} --- which specify a preferred description in the relevant exceptional cases and take the Lie rank to mean the untwisted rank, that is, the rank of the ambient algebraic group.

Moreover, we shall need the following upper bound from \cite{Pyber} (given as a consequence of some tables in \cite{Liebeck}), which is expressed in terms of the parameters $r$ and $q$: no compatibility issue arises, since the authors of \cite{Pyber} explicitly state that it is valid ``for whichever standard definition of these two parameters one cares to take''.

\begin{lemma}[\cite{Pyber}, Proposition 2.2]\label{PyberLemma}
    Let $S$ be a finite simple group of Lie type of rank $r$ and field size $q$. Then
    $$
    |S| \leq q^{8r^2} \, .
    $$
\end{lemma}

\subsection{Almost simple groups, Monolithic groups}
For a finite group $S$, we denote, as usual, by $\operatorname{Aut}(S)$, $\operatorname{Inn}(S)$ and $\operatorname{Out}(S)$, respectively, the group of automorphisms of $S$, the group of inner automorphisms of $S$ and the quotient $\operatorname{Aut}(S)/\operatorname{Inn}(S)$ (see, e.g., \cite{Robinson}).

A finite group $G$ is called almost simple if there exists a non-abelian simple group $S$ such that, up to isomorphism,
$$
S\leq G\leq \operatorname{Aut}(S) \, .
$$
By replacing the given embedding of $G$ into $\operatorname{Aut}(S)$, if necessary, we may assume that $S$ is identified with $\operatorname{Inn}(S)$ via the natural isomorphism induced by conjugation (see 1.5.2 and 1.5.3 in \cite{Robinson}) --- this compatibility is often explicitly required in standard definitions, although it is not necessary to do so in the finite case.
Equivalently, a finite group~$G$ is almost simple if and only if it has a unique minimal non-trivial normal subgroup that is non-abelian simple; in this case, this subgroup is precisely the subgroup~$S$ appearing in the above definition.
In particular, it follows that $S\trianglelefteq G$ and that~$S$ is the socle of~$G$ (that is, the subgroup generated by all minimal non-trivial normal subgroups of~$G$). Accordingly, it is customary to refer to $S$ as the socle of~$G$, and we shall follow this convention.

More generally, a finite group~$G$ is called monolithic if it has a unique minimal \mbox{non-trivial} normal subgroup $N$. In this case, $N$ coincides with the monolith of~$G$, that is, the intersection of all non-trivial normal subgroups of~$G$. Clearly, $N$ is also the socle of~$G$. However, for a monolithic group~$G$, it is customary to refer to~$N$ as the monolith of~$G$. The subgroup~$N$ is characteristically simple and hence, since it is finite, isomorphic to a direct product of finitely many copies of the same finite simple group (see~3.3.15 in~\cite{Robinson}). Clearly, monolithic groups form a broader class that contains almost simple groups.

\begin{definition}[k-generated groups]
Let $k$ be a positive integer. A group $G$ is said to be $k$-generated if it admits a generating set of cardinality at most $k$.

For a finitely generated group $G$, the minimum cardinality of a generating set is denoted by $d(G)$, that is,
$$
d(G)\coloneqq\min\big\{|X| : X\subseteq G,\ \langle X\rangle=G\big\} \, .
$$
Accordingly, $G$ is $k$-generated if and only if $d(G)\leq k$.
\end{definition}

\begin{theorem}[\cite{LucchiniMenegazzo}, Theorem 1.1]\label{LucchMen} 
Let $G$ be a non-cyclic finite group with a unique minimal non-trivial normal subgroup $N$. Then
$$
d(G)=\max\{2,d(G/N)\}.
$$
\end{theorem}

\begin{lemma}\label{d(G)=2}
    Let $G$ be an almost simple group with socle $S$ such that $G/S$ is cyclic. Then $d(G) = 2$. In particular, $G$ is $2$-generated.
\end{lemma}
\begin{proof}
    This follows immediately from Theorem \ref{LucchMen}.
\end{proof}
\begin{theorem}[\cite{Colva}, Lemma 2.1]\label{Colvath}
Let $S$ be a finite non-abelian simple group. Then
$$
\big|\operatorname{Out}(S)\big| \leq \frac{6}{7}\log_2|S| \, .
$$
\end{theorem}

\begin{lemma}\label{ColvaLemma}
    Let $G$ be a finite almost simple group with socle $S$. Then
    $$
    \left|\frac{G}{S}\right| \leq \frac{6}{7}\log_2|S| \, .
    $$
    In particular,
    $$
    \left|\frac{G}{S}\right| < \log_2|S| \, .
    $$
\end{lemma}
\begin{proof}
    Since $G/S \simeq G/\operatorname{Inn}(S)\leq \operatorname{Aut}(S)/\operatorname{Inn(S)} = \operatorname{Out}(S)$, the statement follows immediately from Theorem \ref{Colvath}.
\end{proof}

\subsection{Asymptotic conventions}

Let $\mathcal{A}$ be a family of finite groups, and let
$$
\sigma:\mathcal{A}\longrightarrow \mathbb{R}
\quad \mbox{ and } \quad
f:\mathcal{A}\longrightarrow \mathbb{R}
$$
be real-valued functions. Let $k$ be a real number. We write
$$
\lim_{\sigma(G)\to+\infty} f(G)=k
$$
to mean that, for every $\varepsilon>0$, there exists a positive real number $s_0$ such that, for every $G\in\mathcal{A}$ satisfying $\sigma(G)\geq s_0$, we have
$$
|f(G)-k|<\varepsilon \, .
$$
An analogous definition is adopted when $k=+\infty$: for every $R>0$, there exists a positive real number $s_0$ such that, for every $G\in\mathcal{A}$ satisfying $\sigma(G)\geq s_0$, we have $f(G)>R$.

In this paper, the families $\mathcal{A}$ under consideration will consist of almost simple groups satisfying certain hypotheses, and $\sigma(G)$ will be the order of the socle $S$ of $G$. Accordingly, we shall write
$$
\lim_{|S|\to+\infty} f(G)=k \, ,
$$
where $f(G)$ will be replaced by the explicit expression of the function under consideration.

Similarly, whenever we say that a statement holds ``for~$|S|$ sufficiently large'', ``for~$|S|$ large enough'', or use a similar formulation, we mean that there exists a positive real number $s_0$ such that the statement holds for every almost simple group $G$ in the family under consideration whose socle $S$ satisfies $|S|\geq s_0$. For our purposes, it will be fundamental to have a single threshold $s_0$ that works simultaneously for all groups in the family.

\subsection{Probabilistic estimates for group generation}

To prove Theorem A, we rely on some powerful results that turn certain $2$-generation properties of a quotient into probabilistic $2$-generation estimates for the original group. More precisely, we shall use certain lower bounds regarding two different probabilities associated with the event that two elements, lying in two fixed cosets, generate the original group, \textit{given that those cosets jointly generate the corresponding quotient}. The details of this subsection will clarify the meaning of this sentence.
\\[-5pt]

The first lower bound concerns a ``pointwise'' $2$-generation probability: the first element is fixed, while the second one is chosen uniformly at random from the prescribed coset. The result we need follows immediately from Theorem 1.1 of \cite{FULMAN}, which we recall below.

\begin{theorem}[\cite{FULMAN}, Theorem 1.1]\label{FULMANth}
There exists an absolute constant $\delta>0$ such that the following holds.
Let $S$ be a finite simple group of Lie type of sufficiently large order, and let $x,y\in \operatorname{Aut}(S)$ with $x\neq 1$. Then, under the natural identification of $S$ with $\operatorname{Inn}(S)$ induced by conjugation, the probability that $x$ and a random element of $yS$ generate $\langle S,x,y\rangle$ is at least $\delta$.
\end{theorem}

\begin{coroll}\label{FULMANcor}
    There exists an absolute constant $\delta >0$ such that the following holds. Let $G$ be a finite almost simple group with socle $S$ of Lie type of
    sufficiently large order, and let $x, y \in G$, with $x\neq 1$, such that
    $$
    \langle xS, yS \rangle = \frac{G}{S} \, .
    $$
    Then the probability that $x$ and a random element of $yS$ generate $G$ is at least $\delta$. Equivalently,
    $$
    \big| \{z\in yS : \langle x, z \rangle = G  \} \big|  \geq \delta|S| \, .
    $$
    Note that, since $G$ is not cyclic, in terms of the generating graph of $G$, using the notation introduced in Definition~\ref{degree}, the last inequality can be expressed as
    $$
    d_{yS}(x) \geq \delta|S| \, .
    $$
\end{coroll}
\begin{proof}
    Since $G$ is almost simple with socle $S$, we may regard $G$ as a subgroup of $\operatorname{Aut}(S)$, with $S$ identified with $\operatorname{Inn}(S)$ via the natural isomorphism induced by conjugation. By assumption, $xS$ and $yS$ generate $G/S$, so $G = \langle x, y, S \rangle$. Thus, the statement holds with the same constant of Theorem \ref{FULMANth}.
\end{proof}

In the framework of generating graphs, Corollary~\ref{FULMANcor} tells us that, for an almost simple group, once the order of the socle exceeds a certain threshold, every vertex is adjacent to (at least) a certain portion of the vertices of each suitable coset, where this portion is given by a universal constant.\\[-8pt]

\noindent For the rest of the paper $\delta$ will be the constant defined by Corollary~\ref{FULMANcor}.\\[1pt]

The second lower bound, specific to monolithic groups, concerns instead a ``global'' \mbox{$2$-generation} probability: both elements are chosen independently and uniformly at random from their respective prescribed cosets. First, we shall state the general result and then show how to derive from it the specific consequence that we need.
On this point, the background we recall below is essentially developed in~\cite{uniserial}, with an additional argument from~\cite{Gaschutz}.

\begin{definition}[$2$-generation probabilities]
    Let $G$ be a finite group. We denote by~$P_2(G)$ the probability that two elements chosen independently and uniformly at random from~$G$ generate~$G$, namely,
$$
P_2(G) \coloneqq \frac{ \big|\{(x,y)\in G^2 :  \langle x, y \rangle = G \} \big| }{|G|^2} \, .
$$
Moreover, if $S\trianglelefteq G$ and $P_2(G/S)>0$, we define
$$
P_2(G,S)\coloneqq \frac{P_2(G)}{P_2(G/S)} \, ,
$$
which can be regarded as the conditional probability that two elements chosen independently and uniformly at random from~$G$ generate~$G$, given that their natural images in~$G/S$ generate~$G/S$.

More generally, for any integer $d\geq 2$, one may define $P_d(G)$ and $P_d(G,S)$ analogously, by considering $d$-tuples of elements instead of pairs.
\end{definition}

\begin{definition}[The functions $\alpha$ and $\beta$ --- \cite{uniserial}, Section 2.3]
    Let $S$ be a finite simple group of Lie type of rank $r$ and field size $q$.
    Following \cite{uniserial}, let $\alpha$ be the positive real-valued function defined by
    $$
    \alpha(S) \coloneqq \left( q^{-r/15}+\frac{61}{15}|S|^{-3/4} \right)^{-1/2} \, .
    $$
    Moreover, for convenience, we also define a positive real-valued function $\beta$ by setting
    $$
    \beta(S) \coloneqq \big[ \alpha(S) \big]^{-1} \, .
    $$
    For every $R>0$, there are only finitely many pairs $(q,r)$ such that $q^r\leq R$, and each such pair gives rise to only finitely many finite simple groups $S$ of Lie type with rank $r$ and field size $q$. Therefore, $q^r$ can be made arbitrarily large by taking $|S|$ sufficiently large. Consequently,
    $$
    \lim_{|S|\to +\infty} \alpha(S) = +\infty \qquad \mbox{and} \qquad \lim_{|S|\to +\infty} \beta(S) = 0 \, .
    $$
    For completeness, we mention that the function $\alpha$ considered in \cite{uniserial} is defined for every finite simple group. Since we shall use it only for simple groups of Lie type, we do not need to recall its definition in the other cases.
\end{definition}

\begin{theorem}[\cite{uniserial}, Corollary 9]\label{HQstrongTh}
Let $d\geq 2$ be an integer. Let $G$ be a finite $d$-generated group with a unique minimal non-trivial normal subgroup $N$. Thus, we can write $N=S^n$, with $S$ a finite simple group. Then
$$
P_d(G,N) > 1 - \big[\alpha(S)\big]^{-n(d-\iota)} \, ,
$$
where $\iota=3/2$ if $S$ is abelian and $\iota =1$ otherwise.
\end{theorem}

\begin{coroll}\label{beta}
    Let $G$ be a finite almost simple group with socle $S$ of Lie type. Then
    $$
    P_2(G,S) > 1-\beta(S) \, .
    $$
\end{coroll}
\begin{proof}
By Lemma~\ref{d(G)=2}, $G$ is $2$-generated. Therefore, the statement follows from Theorem~\ref{HQstrongTh}, whose parameters under the present hypotheses are $n=1$, $N=S$, $d=2$, and~\mbox{$\iota=1$.}
\end{proof}

\begin{lemma}[\cite{Gaschutz},
see proof of Satz 1]\label{LemmaGaschutz1}
Let $G$ be a group and $S$ a finite normal subgroup of $G$. Assume that $G/S$ is $2$-generated. Then there exists a non-negative integer constant $\lambda$ such that, for every generating pair $(aS,bS)$ of $G/S$, one has
$$
\big|\{(x,y)\in aS\times bS : \langle x,y \rangle = G \}\big| = \lambda \, .
$$
In other words, this cardinality is independent of the particular generating pair $(aS,bS)$ of $G/S$ considered.
\end{lemma}
\begin{proof}
    This fact is established in the proof of Satz~1 in \cite{Gaschutz}. Indeed, Gaschütz proves it in order to derive the statement of Satz~1.
\end{proof}

\begin{coroll}\label{conditionalprobability}
    Let $G$ be a finite group and $S\trianglelefteq G$. Assume that $G/S$ is $2$-generated. Then, for every generating pair $(aS,bS)$ of $G/S$, one has
$$
P_2(G,S) = \frac{ \big| \{(x,y)\in aS\times bS : \langle x, y \rangle = G \} \big| }{|S|^2} \, .
$$
\end{coroll}
\begin{proof}
Using Lemma \ref{LemmaGaschutz1}, we have
$$
\begin{array}{l}
\displaystyle \big|\{(x,y)\in G^2 : \langle x,y\rangle = G\}\big| \, \, = \sum_{\substack{\alpha,\beta\in G/S\\ \langle \alpha, \beta \rangle = G/S}} \big|\{ (x,y)\in \alpha\times \beta : \langle x,y\rangle = G \}\big| \, =\\[35pt]
\displaystyle = \sum_{\substack{\alpha,\beta\in G/S\\
\langle \alpha, \beta \rangle = G/S}} \!\!\!\! \lambda \,  = \, \displaystyle \lambda \cdot \big|\{(\alpha,\beta)\in (G/S)^2 : \langle \alpha, \beta\rangle = G/S\} \big| \, .
\end{array}
$$
Therefore, we obtain
$$\begin{array}{l}
\displaystyle
P_2(G,S) \, = \, \frac{P_2(G)}{P_2(G/S)} \, = \, \dfrac{\dfrac{\big|\{(x,y)\in G^2 : \langle x,y\rangle = G\}\big|}{|G|^2}}{\dfrac{\big|\{(\alpha,\beta)\in (G/S)^2 : \langle \alpha, \beta\rangle = G/S\} \big|}{|G/S|^2}} \, =\\[35pt]
\displaystyle = \, \frac{\lambda \cdot \big|\{(\alpha,\beta)\in (G/S)^2 : \langle \alpha, \beta\rangle = G/S\} \big|}{\big|\{(\alpha,\beta)\in (G/S)^2 : \langle \alpha, \beta\rangle = G/S\} \big|\cdot |S|^2} \, = \, \frac{\lambda}{|S|^2} \, =\\[35pt]
\displaystyle = \, \frac{\big|\{(x,y)\in aS\times bS : \langle x, y \rangle = G \}\big|}{|S|^2} \, ,
\end{array}
$$
for any fixed generating pair $(aS, bS)$ of $G/S$.
\end{proof}
    
\noindent We are now ready to state the specific consequence that we shall need in our proof.

\begin{coroll}\label{coset-beta}
    For any fixed pair $(aS, bS)$ of elements generating $G/S$, we have
$$
\frac{\big|\{(x,y)\in aS\times bS : \langle x, y \rangle = G \}\big|}{|S|^2} > 1 -\beta(S) \, ,
$$
or, equivalently,
 $$
\frac{\big|\{(x,y)\in aS\times bS : \langle x, y \rangle \neq G \}\big|}{|S|^2} < \beta(S) \, .
$$
\end{coroll}
\begin{proof}
    It is an immediate consequence of Lemma \ref{beta} and Corollary \ref{conditionalprobability}.
\end{proof}

\section{The proof of Theorem A}
Throughout this section, $G$ denotes a finite almost simple group with socle $S$ of Lie type such that $G/S$ is cyclic. Moreover, we shall assume that $|G|$ is as large as needed. This is equivalent to assuming that $|S|$ is as large as needed. Indeed, if $|S|$ were bounded, then there would be only finitely many possibilities for $S$, and, since $G\leq \operatorname{Aut}(S)$, $|G|$ would also be bounded. Conversely, it suffices to note that $|S|\leq |G|$. Under these assumptions, we shall prove that $\Gamma(G)$ has a Hamiltonian cycle, thus establishing Theorem A.\\[-5pt]

\noindent \textbf{Idea of the proof.} The engine of the proof is the Harper-Quick estimate given by Corollary~9 in~\cite{uniserial}, which in our setting takes the form of Corollary~\ref{coset-beta}.
This powerful result establishes that, in the generating graph~$\Gamma(G)$, provided that $S$ is large enough, there is a high density of edges between any two cosets that generate~$G/S$. The main idea is to exploit this high density to construct a Hamiltonian cycle explicitly. More precisely, after fixing an ordered sequence of cosets such that each pair of consecutive cosets generates~$G/S$, we would like to construct such a cycle by taking vertices cyclically from these cosets, in the prescribed order. Unfortunately, this cannot be done quite so directly. To achieve our goal, we exploit the high edge density to find perfect matchings between consecutive cosets in the sequence. By juxtaposing these perfect matchings, we naturally obtain a collection of disjoint paths, made up of one vertex from each coset in the prescribed order. Then, joining these paths leads to the construction of a Hamiltonian cycle.
The natural tool for finding these perfect matchings is Hall's marriage theorem. However, high edge density alone is not sufficient, since an application of Hall's marriage theorem also requires a certain pointwise regularity of adjacency. We therefore first remove from the cosets those vertices whose adjacency behaviour is not sufficiently regular, thus distinguishing between bad and good vertices. Then, the bad vertices are incorporated into a dedicated path~$P_B$, whose construction is made possible by the minimum pointwise regularity of adjacency guaranteed by Corollary~\ref{FULMANcor} (consequence of the Fulman-Garzoni-Guralnick Theorem 1.1 in~\cite{FULMAN}). Next, the remaining good vertices are organized into a path~$P_K$, constructed by means of Hall's marriage theorem. Finally, the two paths are joined together, yielding the desired Hamiltonian cycle.

\subsection{Distinguished pairs of cosets}
By hypothesis, $G/S$ is cyclic, so choose $g\in G$ such that
$$
G/S = \langle gS \rangle \, ;
$$
also, let
$$
n \coloneqq |G/S| \, .
$$
Hence, $G$ is the disjoint union of the cosets $g^0S=S, g^1S= gS, g^2S, \ldots, g^{n-1}S$.
For notational convenience, we set
$$
C_i \coloneqq g^iS \, ,
$$
for every $i\in\{0,1,\ldots, n-1\}$, so
$$
G = C_0 \sqcup C_1 \sqcup \ldots \sqcup C_{n-1} \, .
$$
Note that $C_0=S$. Since the identity is not a vertex of $\Gamma(G)$, we have
$$
V\big(\Gamma(G)\big) = \big(C_0\setminus\{1\}\big) \sqcup C_1 \sqcup \ldots \sqcup C_{n-1} \, ,
$$
where $V\big(\Gamma(G)\big)$ is the set of vertices of $\Gamma(G)$.

Given a non-empty set $A$ and elements $a,b\in A$, not necessarily distinct, we shall denote by $\big[(a,b)\big]$ the unordered pair consisting of $a$ and $b$ (formally, $\big[(a,b)\big]$ is the equivalence class of $(a,b)$ in the quotient of $A^2$ by the equivalence relation that identifies $(x,y)$ with $(y,x)$ for all $x,y\in A$).

\begin{sloppypar}
Now, we single out certain unordered pairs of cosets that generate $G/S$ (the ordering is irrelevant for our purposes). 
We shall refer to the following unordered pairs of cosets as \textit{the distinguished pairs (of cosets)}:
\begin{itemize}
   \item[$\scriptstyle\bullet$] $\big[(C_i,C_{i+1})\big]$ for each  $i\in\{0,1,\ldots,n-2\}$,
    \item[$\scriptstyle\bullet$] $\big[(C_{n-1},C_0)\big]$,
    \item[$\scriptstyle\bullet$] $\big[(C_1,C_{n-1})\big]$.
\end{itemize}
They are all generating pairs of $G/S$ (the first family because $(i,i+1,n)=1$, while the remaining two pairs because $C_{n-1}$ is a generator), hence, the estimates in Corollaries~\ref{FULMANcor} and~\ref{coset-beta} apply to them. Note that if $n=2$, then the only distinguished pairs are $\big[(C_0,C_1)\big]$ and $\big[(C_1,C_1)\big]$, and this is the only case where there is a distinguished pair with a same coset repeated.
\end{sloppypar}

\subsection{Good vertices and Bad vertices}
In this subsection, we identify a subset of ``bad vertices'' of $\Gamma(G)$, consisting of those vertices that weaken the internal connections of at least one distinguished pair. By contrast, we refer to the remaining vertices as ``good vertices''. Moreover, we prove that, as the size of the group grows, the bad vertices form an arbitrarily small fraction of all vertices. In the following subsections, we shall then handle these two types of vertices separately in order to construct a Hamiltonian cycle in $\Gamma(G)$.\\[-5pt]

\begin{definition}[Good vertices and Bad vertices]
Let $\delta$ be the universal positive constant defined by Corollary \ref{FULMANcor}, and choose a constant $\mu$ such that:
$$
0 < \mu < \min\left\{\delta, \frac{1}{4} \right\} \, .
$$

\begin{sloppypar}
Let $x\in C_i$ for some $i\in\{0,1,\ldots,n-1\}$, with $x\neq 1$.
In the generating graph $\Gamma(G)$, $x$ is called a \textit{bad vertex} if there exists a coset $C_j$, for some $j\in\{0,1,\ldots,n-1\}$, such that $\big[(C_i,C_j)\big]$ is a distinguished pair and
$$
d_{C_j}(x) \leq (1-\mu)|C_j| = (1-\mu)|S| \, .
$$
For any such $j$, we also say that $x$ is a bad vertex towards $C_j$.

\noindent Otherwise, $x$ is said to be a \textit{good vertex}, meaning that for every coset $C_j$, with \mbox{$j\in \{0,1,\ldots,n-1\}$,} such that $\big[(C_i, C_j)\big]$ is a distinguished pair, we have
$$
d_{C_j}(x) > (1-\mu)|C_j| = (1-\mu)|S| \, .
$$
Since $\mu$ is to be regarded as a small positive constant, this definition selects as good those vertices adjacent in $\Gamma(G)$ to almost every vertex across each distinguished pair involving their own coset. Moreover, the loss of degree is less than~$\mu|S|$, so the choice $\mu<\delta$, together with Corollary~\ref{FULMANcor}, ensures that a good vertex of $C_i$ and a bad vertex of~$C_j$ have a common neighbour in a third coset~$C_k$, whenever the pairs $\big[(C_i,C_k)\big]$ and $\big[(C_j,C_k)\big]$ are distinguished. We shall need this fact in Subsection~\ref{BpathSubSec}.

Denote by $B$ the set of bad vertices and by $K$ the set of good vertices.
\end{sloppypar}
\end{definition}

\begin{lemma}\label{BetaBehaviorGen}
    Let $\beta$ be the function defined in Lemma \ref{beta}, and let $M$ be a positive integer. Then
    $$
    \lim_{|S|\to +\infty} n^M\beta(S) = 0 \, .
    $$
\end{lemma}
\begin{proof}
Let $r$ and $q$ be, resp., the rank and the field size associated to the finite simple group of Lie type $S$. We have
$$
\beta(S) = \left( q^{-r/15}+\frac{61}{15}|S|^{-3/4} \right)^{1/2} \leq \left( q^{-r/15} \right)^{1/2} +\left( \frac{61}{15}|S|^{-3/4} \right)^{1/2} = q^{-r/30} + \sqrt{\frac{61}{15}} |S|^{-3/8} \, ,
$$
so
$$
n^M\beta(S) = n^Mq^{-r/30} + n^M\sqrt{\frac{61}{15}} |S|^{-3/8} \, .
$$
We now analyze separately the two addends of the right-hand side.

For the first addend, by Lemmas \ref{PyberLemma} and \ref{ColvaLemma}, we have
$$\begin{array}{ll}
n^M q^{-r/30} & \leq \big(\log_2|S|\big)^Mq^{-r/30} \leq \big(\log_2 q^{8r^2}\big)^Mq^{-r/30} = 8^M \big(r^2\log_2 q)^M \, 2^{-(r\log_2 q)/30} \\[20pt]
\ & \leq 8^M (r\log_2 q)^{2M} \, 2^{-(r\log_2 q)/30} \quad \displaystyle \xrightarrow[|S| \to +\infty ]{} \quad 0 \, ,
\end{array}
$$
since
$$
\lim_{|S|\to +\infty} r\log_2 q = +\infty \quad \mbox{ and } \quad \lim_{x\to +\infty} x^{2M}\, 2^{-x/30} = 0 
$$

\noindent(if $r\log_2 q$ were bounded above by a positive constant $C$, then $\log_2(q^{8r^2})=8r^2\log_2q$ would be bounded above by $8C^2$, and so $|S| \leq q^{8r^2} \leq 2^{8C^2}$, a contradiction).

For the second addend, by Lemma \ref{ColvaLemma} we have
$$
n^M\sqrt{\frac{61}{15}} |S|^{-3/8} \leq \sqrt{\frac{61}{15}}\big(\log_2|S|\big)^M \, |S|^{-3/8} \xrightarrow[|S| \to +\infty ]{} \quad 0 \, .
$$
The statement follows.
\end{proof}

\begin{lemma}\label{estimationBgen}
    The following estimate holds:
    $$
    |B| \leq  \frac{2(n+1)\beta(S)}{\mu}|S| \,  .
    $$
    In particular, 
    $$
    \lim_{|S| \to +\infty } \frac{|B|}{|S|} = 0 \, .
    $$
    In other words, asymptotically as $|S| \to +\infty$, the number of bad vertices is negligible compared to the size of $S$.
\end{lemma}
\begin{proof}
    Let $\big[(C_i, C_j)\big]$ be a distinguished pair of cosets, for some $i, j\in\{0,1,\dots,n-1\}$. 
    
First, Corollary~\ref{coset-beta} gives an upper bound on the number of pairs in
$C_i\times C_j$ that do not generate $G$, namely
    $$
    \big|\{(a,b)\in C_i\times C_j : \langle a, b \rangle \neq G\}\big| \leq \beta(S)|S|^2 \, .
    $$
    
    Denote by $B_{i\to j}$ the set of vertices of $C_i$ that are bad towards $C_j$. Let $x\in B_{i\to j}$, so, by definition,
    $$
    d_{C_j}(x) \leq (1-\mu)|C_j| \, .
    $$
    \begin{sloppypar}
    Thus, as pointed out in Definition \ref{degree}, since $G$ is non-cyclic, we have the following estimate for the \mbox{non-generating} pairs in $\{x\}\times C_j$:
    $$
    \big| \{(x,a)\in\{x\}\times C_j : \langle x, a \rangle \neq G  \} \big| = |C_j| - d_{C_j}(x) \geq |C_j| - (1-\mu)|C_j| = \mu|C_j| = \mu|S| \, .
    $$       
    \end{sloppypar}
    
    Altogether, each bad vertex of $C_i$ towards $C_j$ contributes with at least $\mu |S|$ pairs of $C_i\times C_j$
    that do not generate $G$, whose total number is at most $\beta(S)|S|^2$.
    Therefore, 
    $$
    \big| B_{i\to j} \big| \cdot \mu |S| \leq \beta(S)|S|^2 \, ,
    $$
    so
    $$
    \big| B_{i\to j} \big| \leq \frac{\beta(S)}{\mu} |S| \, .
    $$
    
As a final step, since there are at most $n+1$ distinguished pairs of cosets (also when $n=2$), and each pair might contribute bad vertices in both directions, that is, vertices of either coset which are bad towards the other, we obtain
$$
|B| = \Big|\bigcup B_{i\to j}\Big| \leq \sum \big|B_{i\to j}\big| \leq  \frac{2(n+1)\beta(S)}{\mu}|S| \,  ,
$$
which is the bound of the statement.

\begin{sloppypar}
The last assertion is an immediate consequence of Lemma \ref{BetaBehaviorGen}, applied with~\mbox{$M=1$.}     
\end{sloppypar}
\end{proof}

\subsection{The bad-vertex path $P_B$}\label{BpathSubSec}

In this subsection, we shall identify, through an explicit construction, a path
in $\Gamma(G)$ that contains all the bad vertices (and some good vertices) and is balanced with respect
to the cosets, meaning that it contains the same number of vertices from each
coset. This path will be denoted by $P_B$.

\begin{lemma}\label{badpath}.
There exists a path $P_B$ in $\Gamma(G)$ satisfying all of the following conditions:
\begin{itemize}
    \item $P_B$ contains all the bad vertices (that is, all the elements of $B$);
    \item $P_B$ contains the same number $T$ of vertices from each of the cosets
    $C_0,C_1,\ldots,C_{n-1}$;
    \item the endvertices of $P_B$ are good vertices, one belonging to~$C_0$ and the other to~$C_{n-1}$.
\end{itemize}
\end{lemma}
\begin{proof}
Set
$$
\rho \coloneqq \delta - \mu > 0 \, ,
$$
and
$$
T \coloneqq 2|B|+1 \, .
$$

\begin{sloppypar}
The path $P_B$ that we shall identify will consist of $T$ consecutive strips, say $Q_1, Q_2, \ldots, Q_T$, each made up of $n$ vertices, one from each of the cosets $C_0,\ldots,C_{n-1}$, in that order. Thus, each edge of the path will either join two consecutive vertices within the same strip, or join the last vertex of one strip to the first vertex of the next strip. Accordingly, the intended structure for the path $P_B$ can be represented schematically as follows:
\end{sloppypar}

\vspace{0.3cm}
\newcommand{\hconnect}[2]{%
  \draw[-{Latex[length=2mm]}]
    ($(#1.east)!0.5!(#2.west)+(-7mm,0)$)
    --
    ($(#1.east)!0.5!(#2.west)+(7mm,0)$);
}

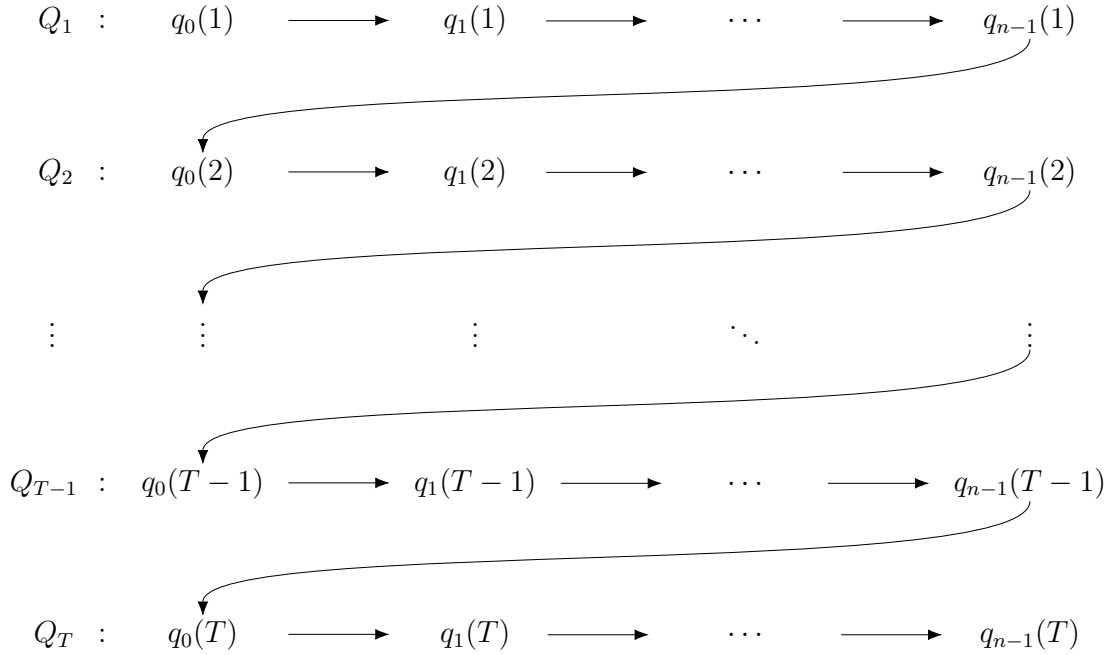
\newcommand{\rowconnect}[2]{%
  \draw[-{Latex[length=2mm]}]
    (#1.south)
    .. controls +(0,-11mm) and +(0,11mm) ..
    (#2.north);
}

\begin{center}
\hspace{-0.6cm}
\begin{adjustbox}{max width=0.95\linewidth}
\begin{tikzpicture}[
  every node/.style={inner sep=1pt}
]

\matrix (M) [
  matrix of math nodes,
  row sep=16mm,
  column sep=20mm
] {
  |[name=q1zero]| q_0(1)
  & |[name=q1one]| q_1(1)
  & |[name=q1dots,inner xsep=6mm]| \cdots
  & |[name=q1end]| q_{n-1}(1)
\\
  |[name=q2zero]| q_0(2)
  & |[name=q2one]| q_1(2)
  & |[name=q2dots,inner xsep=6mm]| \cdots
  & |[name=q2end]| q_{n-1}(2)
\\
  |[name=leftdots]| \vdots
  & \vdots
  & \ddots
  & |[name=rightdots]| \vdots
\\
  |[name=qTm1zero]| q_0(T-1)
  & |[name=qTm1one]| q_1(T-1)
  & |[name=qTm1dots,inner xsep=6mm]| \cdots
  & |[name=qTm1end]| q_{n-1}(T-1)
\\
  |[name=qTzero]| q_0(T)
  & |[name=qTone]| q_1(T)
  & |[name=qTdots,inner xsep=6mm]| \cdots
  & |[name=qTend]| q_{n-1}(T)
\\
};

\coordinate (labelcolumn) at ($(M.west)+(-4mm,0)$);

\node[
  anchor=base east,
  text width=20mm,
  align=right,
  inner sep=0pt
] at (labelcolumn |- q1zero.base)
  {$Q_1\,:$};

\node[
  anchor=base east,
  text width=20mm,
  align=right,
  inner sep=0pt
] at (labelcolumn |- q2zero.base)
  {$Q_2\,:$};

\node[
  anchor=base east,
  text width=15mm,
  align=center,
  inner sep=0pt
] at (labelcolumn |- leftdots.base)
  {$\vdots$};

\node[
  anchor=base east,
  text width=20mm,
  align=right,
  inner sep=0pt
] at (labelcolumn |- qTm1zero.base)
  {$Q_{T-1}\,:$};

\node[
  anchor=base east,
  text width=20mm,
  align=right,
  inner sep=0pt
] at (labelcolumn |- qTzero.base)
  {$Q_T\,:$};

\hconnect{q1zero}{q1one}
\hconnect{q1one}{q1dots}
\hconnect{q1dots}{q1end}

\hconnect{q2zero}{q2one}
\hconnect{q2one}{q2dots}
\hconnect{q2dots}{q2end}

\hconnect{qTm1zero}{qTm1one}
\hconnect{qTm1one}{qTm1dots}
\hconnect{qTm1dots}{qTm1end}

\hconnect{qTzero}{qTone}
\hconnect{qTone}{qTdots}
\hconnect{qTdots}{qTend}

\rowconnect{q1end}{q2zero}
\rowconnect{q2end}{leftdots}
\rowconnect{rightdots}{qTm1zero}
\rowconnect{qTm1end}{qTzero}

\end{tikzpicture}
\end{adjustbox}
\end{center}

\vspace{0.3cm}

\noindent At this stage, the symbols $q_i(k)$ are placeholders: $q_i(k)$ will denote the vertex chosen from $C_i$ to occupy the $i$-th position of the $k$-th strip. Formally, constructing $P_B$ amounts to specifying these vertices. Now we do so, thereby filling the strips with all the bad vertices and some good vertices, while ensuring that the resulting sequence satisfies the requirements of the statement.
We first place the bad vertices, and then look for suitable good vertices to fill the remaining positions. The difficulty is that, with the probability estimates at our disposal, it is not always possible to find a good vertex adjacent to two distinct bad vertices. For this reason, we shall not place bad vertices too close to each other.
More precisely, we shall place exactly one bad vertex in each even-indexed strip, while odd-indexed strips will contain only good vertices. Since each strip contains $n\geq 2$ vertices, this arrangement guarantees that any two consecutive bad vertices along the path will be separated by at least two positions reserved for good vertices.

Fix an enumeration:
$$
B = \{b_1, b_2, \dots, b_{|B|}\} \, . 
$$
For each $k\in\{1,\ldots,|B|\}$, let $i(k)$ be the unique index such that $b_k\in C_{i(k)}$, and place $b_k$ in the $i(k)$-th position of the strip $Q_{2k}$; more formally, set $q_{i(k)}(2k) \coloneqq b_k$. 

In all the remaining positions, across all the strips, we have to place distinct good vertices in such a way that any two vertices occupying consecutive positions are joined by an edge of~$\Gamma(G)$: they may both be good vertices, or one may be a good vertex and the other may be one of the bad vertices already placed.
We begin by placing an arbitrarily chosen good vertex of~$C_0\setminus\{1\}$ in the first position (the $0$-th position) of the first strip~$Q_1$: for~$|S|$ large enough, Lemma \ref{estimationBgen} yields
$$|C_0\setminus\{1\}| - |B| = |S|-1-|B| > 0 \, ,$$
so such a vertex exists.
We then proceed by placing further good vertices, in order from left to right, in the subsequent empty positions encountered along the path, as described in the following. Suppose that we have to fill the $i$-th position of the $k$-th strip.
Thus, we have to choose a suitable vertex in $C_i$. The first requirement is that this vertex is adjacent to the vertex immediately preceding it and also to the vertex immediately following it (if already placed). Denote by $U_{i,k}$ the set of vertices of $C_i$ satisfying these adjacency conditions, ignoring for the moment the additional requirements that the chosen vertex must be good and distinct from all the vertices already placed. Three situations may occur, depending on where the position to be filled lies. We shall prove separately that, in each case, we have
$$
|U_{i,k}| \geq \rho|S| \, .
$$

\noindent For the arguments that follow, a key observation is that the structure prescribed for $P_B$ ensures that any two consecutive vertices along the path belong to cosets that form a distinguished pair.

\begin{enumerate}[label=(\Roman*)]
    \item {\bf The position to be filled follows a good vertex~$x$ and precedes another empty position.}\\
    Since $x$ is a good vertex, by definition
    $$
    d_{C_i}(x) > (1-\mu)|C_i| \, ,
    $$
    that is, there are strictly more than $(1-\mu)|C_i|$ vertices adjacent to $x$ in $C_i$. Thus, since~$\delta \leq 1$,
    $$
    |U_{i,k}| = d_{C_i}(x) > (1-\mu)|C_i| = (1-\delta+\rho)|S| \geq \rho|S| \, . 
    $$
    \item {\bf The position to be filled follows a bad vertex~$y$.}\\
    Note that, in this case, the prescribed spacing of the bad vertices ensures that the position to be filled precedes another empty position. By Corollary \ref{FULMANcor}, we have
    $$
    |U_{i, k}| = d_{C_i}(y) \geq \delta|S| > (\delta - \mu)|S| = \rho|S| \, . 
    $$

    \item {\bf The position to be filled follows a good vertex~$x$ and precedes a bad vertex~$y$.}\\
    This is the most delicate situation. By Corollary \ref{FULMANcor}, we have
    $$
    d_{C_i}(y) \geq \delta|S| \, ,
    $$
    that is, there are at least $\delta|S|$ vertices adjacent to $y$ in $C_i$. Moreover, since $x$ is a good vertex, as we have $d_{C_i}(x) > (1-\mu)|C_i|$, it follows that the number of vertices non-adjacent to $x$ and belonging to $C_i$ is at most
    $$
    |C_i|-d_{C_i}(x) < |C_i|-(1-\mu)|C_i|= \mu|C_i|=\mu|S| \, . 
    $$ 
    In particular, among the at least~$\delta|S|$ neighbours of~$y$ in~$C_i$, there are fewer than~$\mu|S|$ vertices that are not adjacent to~$x$. Therefore, imposing both adjacency requirements, we obtain:
    $$
    |U_{i,k}| \geq \delta|S| - \mu|S| = \rho|S| \, .
    $$
    
\end{enumerate}

\begin{remark}
    Actually, case (I) could be treated in the same way as case (II), by directly applying Corollary \ref{FULMANcor}. However, in preparation for case (III), it is useful, from an expository point of view, to treat it separately in the way described above.
\end{remark}
Now, we take into account the additional requirements that the vertex to be chosen to fill the position must be good and distinct from all the vertices already placed. Denote by $V_{i,k}$ the union of the set of good vertices already placed before the $i$-th position of the $k$-th strip and the set $B$ of all bad vertices. Clearly, since the vertices already placed are less than the total number $nT$ of vertices of the complete path $P_B$, we have
$$
|V_{i,k}| \leq nT+|B| \, .
$$
By an application of Lemma \ref{estimationBgen}, Lemma \ref{BetaBehaviorGen} with $M=1$ and $M=2$, and Lemma \ref{ColvaLemma}, we obtain
$$
\frac{nT + |B|}{|S|} = \frac{(2n+1)|B|}{|S|} + \frac{n}{|S|} \leq \frac{(4n^2+6n+2)}{\mu}\beta(S) + \frac{\log_2|S|}{|S|} \quad \xrightarrow[|S| \to +\infty ]{} \quad 0 \, ,
$$
hence,
$$
\lim_{|S|\to +\infty} \frac{nT+|B|}{|S|} = 0 \, .
$$
Thus, we can take $|S|$ sufficiently large to have
$$
|V_{i,k}| \leq nT + |B| < \rho|S| \, .
$$
On the other hand, we proved $|U_{i,k}| \geq \rho|S|$, so, by cardinality, we obtain
$$
U_{i,k} \setminus V_{i,k} \neq \emptyset \, .
$$
Therefore, any element of $U_{i,k}\setminus V_{i,k}$ is a valid choice for filling the $i$-th position of the $k$-th strip: it satisfies the required adjacency conditions in $\Gamma(G)$, is a good vertex, and is distinct from all the vertices already placed.

Therefore, the construction of the path $P_B$ can be completed in such a way that all the requirements of the statement are satisfied.
\end{proof}

\begin{sloppypar}
\noindent From now on, we fix a path $P_B$ as described in Lemma~\ref{badpath}.

For each \mbox{$i\in\{0,1,\ldots, n-1\}$,} let $D_i$ be the set of the $T$ vertices of $C_i$ involved in the path $P_B$; more precisely, in accordance with the previous notation,
$$
D_i \coloneqq \{ q_i(1), q_i(2), \ldots, q_i(T) \} \, .
$$
We are naturally led to consider the remaining vertices of each $C_i$, meaning those vertices of $C_i$ that are not involved in the path $P_B$, so we define the sets
$$\begin{array}{l}
E_0 \coloneqq \big(C_0\setminus\{1\}\big) \setminus D_0 \, ,\\[6pt]
E_i \coloneqq C_i \setminus D_i \, , \, \mbox{ for each } i \in \{1,2,\ldots,n-1\} \, .
\end{array}
$$
\end{sloppypar}
\noindent Moreover, we set
$$
m \coloneqq |S| - T \,  ,
$$
so
$$
\begin{array}{l}
|E_0| = m - 1 \, ,\\[6pt]
|E_i| = m \, , \, \mbox{ for each } i \in \{1,2,\ldots,n-1\} \, .
\end{array}
$$

\subsection{The good-vertex path $P_K$}
In this subsection we shall show how to identify a path $P_K$ in $\Gamma(G)$ that passes through all the vertices of the~$E_i$'s. In order to do so, we shall make use of the well-known ``marriage theorem'' of Philip Hall, recalled in the preliminaries. Note that now we are in a situation where we have to deal only with good vertices, so the adjacency between any two $E_i$'s associated to a distinguished pair is uniformly dense, meaning that every vertex of either set has many neighbours in the other set (as the order of the group grows).

\begin{lemma}\label{Tbehav}
    We have
    $$
    \lim_{|S|\to +\infty} \frac{T}{|S|} = 0 \, .
    $$
    As a consequence,
    $$
    \lim_{|S|\to +\infty} m = +\infty \, .
    $$
\end{lemma}
\begin{proof}
    Since $T=2|B|+1$, Lemma \ref{BetaBehaviorGen} yields
     $$
     \frac{T}{|S|} = \frac{2|B|+1}{|S|} = 2\cdot\frac{|B|}{|S|} + \frac{1}{|S|}  \quad \xrightarrow[|S| \to +\infty ]{} \quad 0 \, .
     $$
     Consequently,
     $$
     \lim_{|S|\to +\infty} m = \lim_{|S|\to +\infty} \big[ |S| - T \big] =
     \lim_{|S|\to +\infty} \left[ |S|\left(
     1-\frac{T}{|S|} \right) \right] = +\infty \, .
     $$
\end{proof}

\begin{lemma}\label{GoodVertexVar}
Let $\big[(C_i,C_j)\big]$ be a distinguished pair of cosets, and let $x\in C_i\setminus\{1\}$ be a good vertex (in particular, this holds if $x\in E_i$). Then, for~$|S|$ large enough, in~$\Gamma(G)$ we~have
$$
d_{E_j}(x)>(1-2\mu)|E_j| \, .
$$
Equivalently, the number of vertices of $E_j$ that are not adjacent to $x$ is strictly less than~$2\mu|E_j|$.
\end{lemma}

\begin{proof}
Let $x\in C_i\setminus\{1\}$. Since $x$ is a good vertex, we have
$$
d_{C_j}(x)>(1-\mu)|C_j|=(1-\mu)|S| \, .
$$

In passing from $C_j$ to $E_j$, at most $T$ neighbours of $x$ are removed. Indeed, if $j\neq 0$, exactly $T$ vertices are removed from $C_j$. If $j=0$, the identity element is also removed, but it is not a vertex of $\Gamma(G)$, and hence it does not contribute to the loss of degree. Therefore,
$$
d_{E_j}(x)>(1-\mu)|S|-T \, .
$$
Since $m=|S|-T$, it follows that
$$
d_{E_j}(x)>m-\mu|S| \, .
$$
\begin{sloppypar}
Consequently, we have
$$
d_{E_j}(x)>(1-2\mu)m
$$
provided that $m-\mu|S|>(1-2\mu)m$. The latter inequality is equivalent to \mbox{$|S|<2m=2(|S|-T)$,} and, in turn, to $T/|S|<1/2$, which, for $|S|$ large enough, is true by Lemma~\ref{Tbehav}. Hence,
$$
d_{E_j}(x)>(1-2\mu)m \geq (1-2\mu)|E_j| \, ,
$$
and the main statement follows. Clearly, the stated equivalence follows from
$$|E_j|-d_{E_j}(x) < |E_j| - (1-2\mu)|E_j| = 2\mu |E_j| \, .$$
\end{sloppypar}
\end{proof}

\begin{lemma}\label{goodpath}
Let $u\in E_1$ and $v\in E_{n-1}$ be chosen arbitrarily, with $u\neq v$. There exists a path~$P_K$ in~$\Gamma(G)$ satisfying all of the following conditions:
\begin{itemize}
    \item $P_K$ contains all the vertices of the sets $E_0, E_1, \ldots, E_{n-1}$;
    \item the path $P_K$ starts at $u$ and ends at $v$.
\end{itemize}
\end{lemma}
    
\begin{proof}
We shall describe how to explicitly construct a path $P_K$ in $\Gamma(G)$ satisfying all the conditions of the statement. At the end of our construction, the path $P_K$ will be formed by a sequence of vertices taken one by one, cyclically, from $E_1, E_2, \ldots, E_{n-1}, E_0$, where the first vertex is $u$ and the last vertex is $v$. The intended structure for the path $P_K$ can be represented by the following scheme:

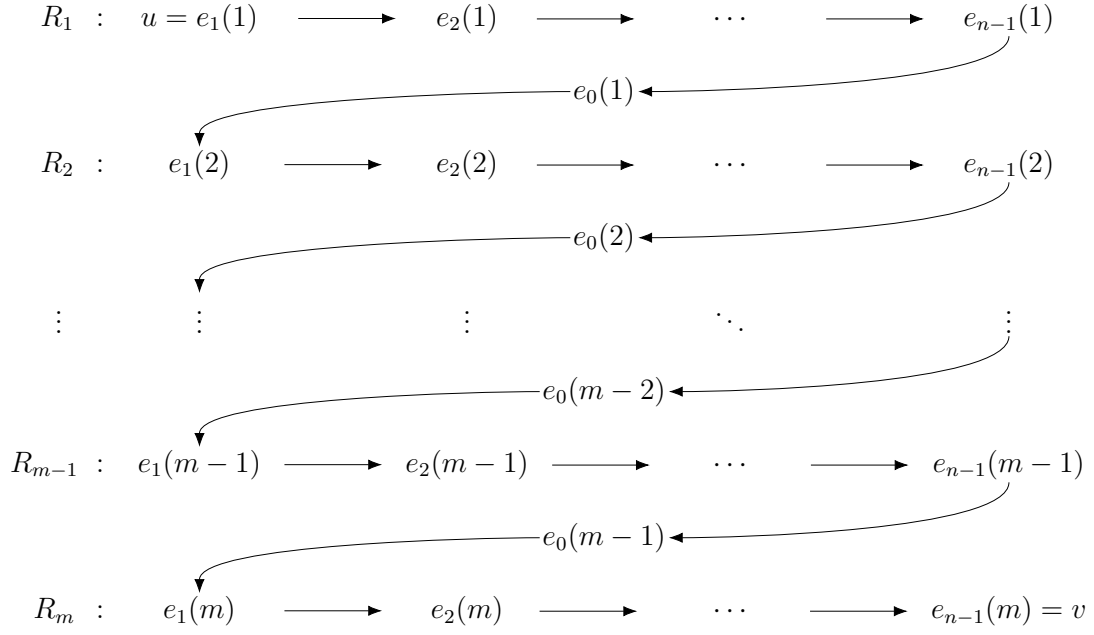
\newcommand{\hconnect}[2]{%
  \draw[-{Latex[length=2mm]}]
    ($(#1.east)!0.5!(#2.west)+(-7mm,0)$)
    --
    ($(#1.east)!0.5!(#2.west)+(7mm,0)$);
}

\newcommand{\splitbridge}[4]{%
  \coordinate (mid-#3) at ($(#1.south)!0.5!(#2.north)$);

  \node[
    anchor=center,
    inner sep=1pt
  ] (#3) at (mid-#3) {$#4$};

  \draw[-{Latex[length=2mm]}]
    (#1.south)
    .. controls +(0,-7mm) and +(9mm,0) ..
    (#3.east);

  \draw[-{Latex[length=2mm]}]
    (#3.west)
    .. controls +(-9mm,0) and +(0,7mm) ..
    (#2.north);
}

\begin{center}
\hspace{-0.8cm}
\begin{adjustbox}{max width=0.92\linewidth}
\begin{tikzpicture}[
  every node/.style={inner sep=1pt}
]

\matrix (M) [
  matrix of math nodes,
  row sep=16mm,
  column sep=20mm
] {
  |[name=r1e1]| u=e_1(1)
  & |[name=r1e2]| e_2(1)
  & |[name=r1dots,inner xsep=6mm]| \cdots
  & |[name=r1end]| e_{n-1}(1)
\\
  |[name=r2start]| e_1(2)
  & |[name=r2e2]| e_2(2)
  & |[name=r2dots,inner xsep=6mm]| \cdots
  & |[name=r2end]| e_{n-1}(2)
\\
  |[name=leftdots]| \vdots
  & \vdots
  & \ddots
  & |[name=rightdots]| \vdots
\\
  |[name=rm1start]| e_1(m-1)
  & |[name=rm1e2]| e_2(m-1)
  & |[name=rm1dots,inner xsep=6mm]| \cdots
  & |[name=rm1end]| e_{n-1}(m-1)
\\
  |[name=rmstart]| e_1(m)
  & |[name=rme2]| e_2(m)
  & |[name=rmdots,inner xsep=6mm]| \cdots
  & |[name=rmend]| e_{n-1}(m)=v
\\
};

\coordinate (labelcolumn) at ($(M.west)+(-4mm,0)$);

\node[
  anchor=base east,
  text width=18mm,
  align=right,
  inner sep=0pt
] at (labelcolumn |- r1e1.base)
  {$R_1\,:$};

\node[
  anchor=base east,
  text width=18mm,
  align=right,
  inner sep=0pt
] at (labelcolumn |- r2start.base)
  {$R_2\,:$};

\node[
  anchor=base east,
  text width=13mm,
  align=center,
  inner sep=0pt
] at (labelcolumn |- leftdots.base)
  {$\vdots$};

\node[
  anchor=base east,
  text width=18mm,
  align=right,
  inner sep=0pt
] at (labelcolumn |- rm1start.base)
  {$R_{m-1}\,:$};

\node[
  anchor=base east,
  text width=18mm,
  align=right,
  inner sep=0pt
] at (labelcolumn |- rmstart.base)
  {$R_m\,:$};

\hconnect{r1e1}{r1e2}
\hconnect{r1e2}{r1dots}
\hconnect{r1dots}{r1end}

\hconnect{r2start}{r2e2}
\hconnect{r2e2}{r2dots}
\hconnect{r2dots}{r2end}

\hconnect{rm1start}{rm1e2}
\hconnect{rm1e2}{rm1dots}
\hconnect{rm1dots}{rm1end}

\hconnect{rmstart}{rme2}
\hconnect{rme2}{rmdots}
\hconnect{rmdots}{rmend}

\splitbridge{r1end}{r2start}{b1}{e_0(1)}
\splitbridge{r2end}{leftdots}{b2}{e_0(2)}
\splitbridge{rightdots}{rm1start}{b3}{e_0(m-2)}
\splitbridge{rm1end}{rmstart}{b4}{e_0(m-1)}

\end{tikzpicture}
\end{adjustbox}
\end{center}
\vspace{0.3cm}

\noindent At this stage, the symbols $e_i(t)$ are placeholders for the vertices that will fill the positions along the path.
As we can see, the path will be formed by $m$ horizontal strips $R_1,R_2,\ldots,R_m$ of vertices, bridged by some reserved vertices.
More precisely, for each $t\in\{1,2,\ldots,m\}$, the strip $R_t$ consists of the $n-1$ vertices
$e_1(t),e_2(t),\ldots,e_{n-1}(t)$: in this order, one from~$E_1$, one from~$E_2$, and so on, up to one from~$E_{n-1}$; clearly, any two consecutive vertices within a strip must be adjacent in $\Gamma(G)$. Finally, for each $t\in\{1,\ldots,m-1\}$, the pair of consecutive strips $(R_t,R_{t+1})$ is bridged by a vertex~$e_0(t)$ of~$E_0$, meaning that~$e_0(t)$ is adjacent in $\Gamma(G)$ to the last vertex of $R_t$ and to the first vertex of $R_{t+1}$. This interpretation essentially follows the same pattern that we shall use to construct the required path.\\

\noindent {\bf Step 1: Construction of the strips $R_t$.}\\
First, assign $e_1(1)\coloneqq u$, and then set $e_1(2),\ldots,e_1(m)$ to be the elements of $D_1\setminus\{u\}$, taken in an arbitrary order.
Now, fix arbitrarily an $i\in\{1,2,\ldots,n-2\}$, and suppose that the elements $e_i(t)$ have already been determined for all $t\in\{1,2,\ldots,m\}$. We will describe how to determine the elements $e_{i+1}(t)$ for all $t\in\{1,2,\ldots,m\}$.

Consider the bipartite subgraph $\Delta$ of $\Gamma(G)$ induced between $E_i$ and $E_{i+1}$ (that is, the graph with vertex set $E_i\cup E_{i+1}$ and edge set consisting of the edges of $\Gamma(G)$ that have an endpoint in $E_i$ and the other in $E_{i+1}$).
Let $x\in E_i$. By Lemma \ref{GoodVertexVar}, we have
$$
d_{E_{i+1}}(x) > (1-2\mu)|E_{i+1}| = (1-2\mu)m > \frac{m}{2} \, ,
$$
where the last inequality follows from $\mu < 1/4$.
Therefore, every vertex of $E_i$ has degree at least $m/2$ in $\Delta$. By a completely symmetric argument, the same holds for every vertex of $E_{i+1}$. Hence, by Corollary~\ref{Hallcor}, there exists a perfect matching in $\Delta$. 
Finally, for each $t\in\{1,\ldots,m\}$, we define $e_{i+1}(t)$ to be the unique vertex adjacent to $e_i(t)$ in this perfect matching.

\begin{sloppypar}
We first carry out this construction for $i=1$, and then repeat it inductively for \mbox{$i=2,\ldots,n-2$.} In this way, all the elements $e_i(t)$, with $i\neq 0$, are defined, and the strips $R_1,R_2,\ldots,R_m$ are determined. Moreover, for every $i\neq 0$ and $t$ in their respective ranges, $e_i(t)$ and $e_{i+1}(t)$ are adjacent in $\Gamma(G)$, as it has to be.
\end{sloppypar}

\mbox{ }\\
{\bf Step 2: Possible switch of the endpoints.}\\
According to our construction, $u$ is the first vertex of the first strip $R_1$. Moreover, for our purposes, we would like the last strip to end at $v$. Up to a suitable permutation of the strips and a corresponding reindexing of all the notation, this can always be achieved, unless $v$ is the last vertex of $R_1$, since in that case the position of $u$ would also change, which we do not want. This step describes how to overcome this obstacle.

Suppose that $v$ is the last element of the first strip $R_1$, that is, $e_{n-1}(1) = v$. For notational convenience, set $x\coloneqq e_{n-2}(1)$, which is the penultimate element of the first strip. 
In order to swap $v$ with the last vertex of another strip, we must choose that strip so that its last vertex (which belongs to $E_{n-1}$) is adjacent to $x$ and its penultimate vertex (which belongs to $E_{n-2}$) is adjacent to $v$. Now, by Lemma~\ref{GoodVertexVar}, there are less than~\mbox{$2\mu|E_{n-1}| = 2\mu m$} vertices of~$E_{n-1}$ that are not adjacent to $x$, and less than~\mbox{$2\mu|E_{n-2}|= 2\mu m$} vertices of~$E_{n-2}$ that are not adjacent to $v$. Thus, less than~$4\mu m$ strips are unsuitable for the swap. Therefore, since there are $m-1$ strips to choose from (after excluding the first strip), an admissible choice is guaranteed to exist if
$$
4\mu m<m-1 \, .
$$
As $\mu<1/4$, the latter inequality is equivalent to
$$
m>\frac{1}{1-4\mu} \, ,
$$
which holds for $|S|$ large enough, by Lemma~\ref{Tbehav}.

In conclusion, if necessary, we may swap $v$ with the last vertex of a suitable other strip. Moreover, as already observed, after reindexing all the notation accordingly, we may assume that $v$ is the last element of the last strip (namely, $e_{n-1}(m)$), while $u$ remains the first element of the first strip (namely, $e_1(1)$).

\mbox{ }\\
{\bf Step 3: Connection of the strips.}\\
In this step, we show how to connect the strips $R_t$ using the $m-1$ elements of $E_0$. For each $t\in I_{m-1}\coloneqq\{1,2,\ldots,m-1\}$, we need to choose an element $e_0(t)\in E_0$ that is adjacent in~$\Gamma(G)$ both to the last vertex of~$R_t$, namely~$e_{n-1}(t)$, and to the first vertex of~$R_{t+1}$, namely~$e_1(t+1)$, in such a way that $t\mapsto e_0(t)$ defines a bijection between $I_{m-1}$ and $E_0$. Denote by $L_t$ the set of vertices of $E_0$ satisfying these two adjacency conditions.
Now, consider the auxiliary graph $\Delta$ with vertex set $I_{m-1}\cup E_0$, in which $t\in I_{m-1}$ is adjacent to $x\in E_0$ if and only if $x\in L_t$ (that is,  $x$ is an admissible choice for $e_0(t)$).  Clearly, $\Delta$ is a bipartite graph with bipartition $(I_{m-1},E_0)$ such that $|I_{m-1}|=|E_0|=m-1$.  We shall use Hall's marriage theorem to show that~$\Delta$ has a perfect matching; more precisely, we shall prove that every vertex of~$\Delta$ has degree at least~$(m-1)/2$, in order to apply Corollary~\ref{Hallcor}.

Let~$t\in I_{m-1}$. Recall that $\big[(C_{n-1},C_0)\big]$ and $\big[(C_0,C_1)\big]$ are distinguished pairs of cosets. Since $e_{n-1}(t)$ is a good vertex, there are strictly less than $\mu|C_0|=\mu|S|$ vertices of $C_0$ that are not adjacent to $e_{n-1}(t)$ in $\Gamma(G)$. As $E_0\subseteq C_0$, the same bound holds for the number of vertices of $E_0$ that are not adjacent to $e_{n-1}(t)$ in $\Gamma(G)$. Analogously, there are strictly less than $\mu|C_0|=\mu|S|$ vertices of $E_0$ that are not adjacent to $e_1(t+1)$ in $\Gamma(G)$.
Therefore,
$$
d_\Delta (t) = |L_t| > |E_0| - 2\mu|S| = m-1 -2\mu|S| \, .
$$
Thus, we obtain
$$
d_\Delta (t) > \frac{m-1}{2} \, ,
$$
provided that $m-1-2\mu|S|>(m-1)/2$, which is true for $|S|$ large enough, by Lemma~\ref{Tbehav}, as $m=|S|-T$ and $\mu<1/4$.

On the other side of the bipartition, let~$x\in E_0$, and let us estimate how many pairs of consecutive strips~$x$ can bridge, that is, equivalently, how many vertices $t\in I_{m-1}$ are adjacent to~$x$ in~$\Delta$. Since~$x$ is a good vertex, using the same argument as above twice, we find that there are strictly less than~$\mu|S|$ vertices of~$E_{n-1}$ that are not adjacent to~$x$ in~$\Gamma(G)$, and that there are strictly less than $\mu|S|$ vertices of $E_1$ not adjacent to~$x$ in~$\Gamma(G)$. Thus, among the $m-1$ pairs of consecutive strips, there are strictly less than~$2\mu|S|$ pairs that~$x$ cannot bridge. 
Therefore,
$$
d_\Delta(x) > m-1 - 2\mu|S| \, .
$$
Hence, as before, for $|S|$ large enough, we obtain
$$
d_\Delta(x) > \frac{m-1}{2} \, .
$$
Therefore, every vertex of~$\Delta$ has degree at least $(m-1)/2$, so, by Corollary~\ref{Hallcor}, $\Delta$ has a perfect matching (between the two parts~$I_{m-1}$ and~$E_0$). By the definition of~$\Delta$, this means that the elements of $E_0$ can be used to connect the strips $R_t$ in the desired manner described above.\\[-10pt]

The proof is now complete.

\end{proof}

We cannot yet fix once and for all a path $P_K$ as described in Lemma~\ref{goodpath}, since the vertices $u$ and $v$ have not yet been chosen. Their choice will be addressed in the following, and final, subsection.

\subsection{Joining the paths $P_B$ and $P_K$}
In this section, we shall complete the construction of a Hamiltonian cycle in $\Gamma(G)$.

First of all, we need to choose two suitable vertices to connect the endpoints of~$P_B$ to the path that, shortly thereafter, will be regarded as the path~$P_K$. As dictated by the third condition of Lemma~\ref{badpath}, let~$p_0\in C_0\setminus\{1\}$ and~$p_{n-1}\in C_{n-1}$ be the good vertices that are the endpoints of~$P_B$.
We seek a vertex of~$E_{n-1}$ adjacent to~$p_0$. Since~$p_0$ is a good vertex of $C_0$ and $\big[(C_0,C_{n-1})\big]$ is a distinguished pair, by Lemma~\ref{GoodVertexVar} we obtain
$$
d_{E_{n-1}}(p_0) > (1-2\mu)|E_{n-1}| = (1-2\mu)m > 0 \, ,
$$
where the last inequality holds for $|S|$ large enough, by Lemma~$\ref{Tbehav}$ and $\mu < 1/4$.
Hence, such a vertex exists: denote it by $v$. 
Analogously, there exists a vertex of $E_1$ adjacent to $p_{n-1}$: denote it by $u$. Note that if $n=2$, then $E_1=E_{n-1}$, and so we must explicitly choose $u$ and $v$ to be distinct: for $|S|$ large enough, clearly, this can be accomplished. 

Now, by Lemma~\ref{goodpath}, there exists a path $P_K$ in $\Gamma(G)$ that starts at $u$, ends at $v$, and passes through all the vertices of $\Gamma(G)$ not involved in $P_B$. Therefore, the cycle obtained by joining the two paths $P_B$ and $P_K$ as described above is a Hamiltonian cycle of $\Gamma(G)$.\\[0pt]

Throughout the proof, several steps are shown to hold for $|S|$ sufficiently large. More precisely, each of these steps holds for every almost simple group $G$ satisfying the hypotheses and whose socle $S$ has order exceeding a certain positive real threshold $s_i$ (which exists, but is not explicitly determined). For the proof to work, all these steps must hold simultaneously: since only finitely many thresholds $s_i$ occur, their maximum provides a single threshold that ensures that all the steps hold.\\[0pt]

This completes the proof of Theorem~A.\hfill\scalebox{1.3}{$\blacksquare$}

\section{Further directions}
As already remarked, the proof of Theorem A given in this paper relies essentially on two powerful ingredients, both valid for almost simple groups: Theorem 1.1 of Fulman-Garzoni-Guralnick in \cite{FULMAN} (restated here as Theorem~\ref{FULMANth}) and Corollary 9 of Harper-Quick in \cite{uniserial} (restated here as Theorem~\ref{HQstrongTh}).
Since the Harper-Quick estimate holds more generally for the much broader class of monolithic groups, it is natural to ask whether Theorem A could also be extended to this class (clearly, in view of Theorem 1.2 of \cite{5authors}, we may restrict our attention to non-abelian monoliths). Indeed, by retracing the arguments used here in the proof of Theorem A, it feels like the only serious obstruction that prevents the same arguments from working for monolithic groups as well is the absence of an analogue of the Fulman-Garzoni-Guralnick theorem for monolithic groups. However, as also observed in the introduction of \cite{FULMAN}, such an extension cannot even hold for the subclass of almost simple groups with alternating socle. The situation for monolithic groups whose monolith is a direct product of simple groups of Lie type or sporadic could be more favourable, although no corresponding result seems to be currently available. Thus, a more reasonable approach would be to find a way around those parts of the proof that make use of the Fulman-Garzoni-Guralnick estimate. Moreover, there is some evidence in the literature that Theorem A could hold for the class of monolithic groups, since it is known to hold for some of its subclasses (see again~\cite{5authors}, and also~\cite{Crestani}).
Therefore, a natural next step in this area would be to establish the following conjecture (which is a special case of Conjecture~\ref{MarotiConj}).

\begin{conj}\label{ConjMon}
Let $G$ be a finite monolithic group with a non-abelian monolith $N$. Write $N=S^k$, where $k$ is a positive integer and $S$ is a non-abelian simple group. Assume that $G/N$ is cyclic. Then, for $S$ sufficiently large, the generating graph $\Gamma(G)$ of $G$ is Hamiltonian. Possibly, the weaker assumption that $G$ (equivalently, $N$) is sufficiently large already suffices.
\end{conj}
Notice that, for $S$ non-abelian, $G$ large enough is equivalent to $N$ large enough. However, $S$ large enough implies both $G$ and $N$ large enough, whereas the converse does not hold. This is another difference that arises in passing from almost simple groups to monolithic groups.\\[-8pt]

Finally, we want to remark that Theorem A is asymptotic in nature, but we currently have no estimate on how large the order of the group must be for its conclusion to hold. Drawing a parallel with the previous literature, the existence of a Hamiltonian cycle for sufficiently large alternating and symmetric groups was proved in 2010, in \cite{5authors}, whereas explicit bounds ensuring the existence of such a cycle were obtained only later, in 2018, in \cite{Erdem}. This naturally leads to the following problem.

\begin{openproblem}
Determine an explicit lower bound, either on the order of the group or on the order of its socle, that ensures the existence of a Hamiltonian cycle for almost simple groups with socle of Lie type.
\end{openproblem}
As a complement to the asymptotic perspective, we mention that for almost simple groups with socle of order at most $10^6$ the existence of a Hamiltonian cycle is established by Theorem~1.5 of \cite{5authors}, whose proof relies on computational methods developed in \cite{BGK}.

\paragraph{Acknowledgements}

This work was developed during the author's research visit to the Alfréd Rényi Institute of Mathematics in Budapest, Hungary. The author thanks the Institute for its hospitality and is grateful to Professor Attila Maróti for introducing him to the topics addressed in this paper and for his valuable suggestions. This work was supported by the “National Group for Algebraic and Geometric Structures, and their Applications” (GNSAGA–INdAM), of which the author is a member. The author is also a member of the non-profit association ``AGTA --- Advances in Group Theory and Applications'' (http://www.advgrouptheory.com).

\bigskip
\bigskip

\begin{minipage}{\textwidth}
\begin{flushleft}
\rule{8cm}{0.4pt}\\
\end{flushleft}

\bigskip
\bigskip

{
\sloppy
\noindent
Luigi Iorio

\noindent 
Dipartimento di Matematica e Applicazioni ``Renato Caccioppoli''

\noindent
Università degli Studi di Napoli Federico II

\noindent
Complesso Universitario Monte S. Angelo

\noindent
Via Cintia, Napoli (Italy)

\noindent
e-mail: luigi.iorio2@unina.it 

}
\end{minipage}

\end{document}